\documentclass[a4paper,twoside,fleqn,notitlepage]{article}%
\usepackage{amsmath}
\usepackage{amsfonts}
\usepackage{amssymb}
\usepackage{graphicx}%
\newtheorem{theorem}{Theorem}

\newtheorem{corollary}[theorem]{Corollary}

\newtheorem{definition}[theorem]{Definition}

\newtheorem{lemma}[theorem]{Lemma}

\newenvironment{proof}[1][Proof]{\noindent\textbf{#1.} }{\ \rule{0.5em}{0.5em}}
\begin{document}

\title{A new class of generalized Bernoulli polynomials and Euler polynomials}
\author{N. I. Mahmudov\\Eastern Mediterranean University\\Gazimagusa, TRNC, Mersiin 10, Turkey \\Email: nazim.mahmudov@emu.edu.tr}
\maketitle

\begin{abstract}
The main purpose of this paper is to introduce and investigate a new class of
generalized Bernoulli polynomials and Euler polynomials based on the
$q$-integers. The $q$-analogues of well-known formulas are derived. The
$q$-analogue of the Srivastava--Pint\'{e}r addition theorem is obtained. We
give new identities involving $q$-Bernstein polynomials.

\end{abstract}

\section{ Introduction}

Throughout this paper, we always make use of the following notation:
$\mathbb{N}$ denotes the set of natural numbers, $\mathbb{N}_{0}$ denotes the
set of nonnegative integers, $\mathbb{R}$ denotes the set of real numbers,
$\mathbb{C}$ denotes the set of complex numbers.

The $q$-shifted factorial is defined by
\[
\left(  a;q\right)  _{0}=1,\ \ \ \left(  a;q\right)  _{n}=%
%TCIMACRO{\dprod \limits_{j=0}^{n-1}}%
%BeginExpansion
{\displaystyle\prod\limits_{j=0}^{n-1}}
%EndExpansion
\left(  1-q^{j}a\right)  ,\ \ \ n\in\mathbb{N},\ \ \ \left(  a;q\right)
_{\infty}=%
%TCIMACRO{\dprod \limits_{j=0}^{\infty}}%
%BeginExpansion
{\displaystyle\prod\limits_{j=0}^{\infty}}
%EndExpansion
\left(  1-q^{j}a\right)  ,\ \ \ \ \left\vert q\right\vert <1,\ \ a\in
\mathbb{C}.
\]
The $q$-numbers and $q$-numbers factorial is defined by%
\[
\left[  a\right]  _{q}=\frac{1-q^{a}}{1-q}\ \ \ \left(  q\neq1\right)
;\ \ \ \left[  0\right]  _{q}!=1;\ \ \ \ \left[  n\right]  _{q}!=\left[
1\right]  _{q}\left[  2\right]  _{q}...\left[  n\right]  _{q}\ \ \ \ \ n\in
\mathbb{N},\ \ a\in\mathbb{C}%
\]
respectively. The $q$-polynomial coefficient is defined by%
\[
\left[
\begin{array}
[c]{c}%
n\\
k
\end{array}
\right]  _{q}=\frac{\left(  q;q\right)  _{n}}{\left(  q;q\right)
_{n-k}\left(  q;q\right)  _{k}}.
\]
The $q$-analogue of the function $\left(  x+y\right)  ^{n}$ is defined by%
\[
\left(  x+y\right)  _{q}^{n}:=%
%TCIMACRO{\dsum \limits_{k=0}^{n}}%
%BeginExpansion
{\displaystyle\sum\limits_{k=0}^{n}}
%EndExpansion
\left[
\begin{array}
[c]{c}%
n\\
k
\end{array}
\right]  _{q}q^{\frac{1}{2}k\left(  k-1\right)  }x^{n-k}y^{k},\ \ \ n\in
\mathbb{N}_{0}.
\]
The $q$-binomial formula is known as%
\[
\left(  1-a\right)  _{q}^{n}=\left(  a;q\right)  _{n}=%
%TCIMACRO{\dprod \limits_{j=0}^{n-1}}%
%BeginExpansion
{\displaystyle\prod\limits_{j=0}^{n-1}}
%EndExpansion
\left(  1-q^{j}a\right)  =\sum_{k=0}^{n}\left[
\begin{array}
[c]{c}%
n\\
k
\end{array}
\right]  _{q}q^{\frac{1}{2}k\left(  k-1\right)  }\left(  -1\right)  ^{k}%
a^{k}.
\]
In the standard approach to the $q$-calculus two exponential functions are
used:%
\begin{align*}
e_{q}\left(  z\right)   &  =\sum_{n=0}^{\infty}\frac{z^{n}}{\left[  n\right]
_{q}!}=\prod_{k=0}^{\infty}\frac{1}{\left(  1-\left(  1-q\right)
q^{k}z\right)  },\ \ \ 0<\left\vert q\right\vert <1,\ \left\vert z\right\vert
<\frac{1}{\left\vert 1-q\right\vert },\ \ \ \ \ \ \ \\
E_{q}\left(  z\right)   &  =\sum_{n=0}^{\infty}\frac{q^{\frac{1}{2}n\left(
n-1\right)  }z^{n}}{\left[  n\right]  _{q}!}=\prod_{k=0}^{\infty}\left(
1+\left(  1-q\right)  q^{k}z\right)  ,\ \ \ \ \ \ \ 0<\left\vert q\right\vert
<1,\ z\in\mathbb{C}.\
\end{align*}
From this form we easily see that $e_{q}\left(  z\right)  E_{q}\left(
-z\right)  =1$. Moreover,%
\[
D_{q}e_{q}\left(  z\right)  =e_{q}\left(  z\right)  ,\ \ \ \ D_{q}E_{q}\left(
z\right)  =E_{q}\left(  qz\right)  ,
\]
where $D_{q}$ is defined by%
\[
D_{q}f\left(  z\right)  :=\frac{f\left(  qz\right)  -f\left(  z\right)
}{qz-z},\ \ \ \ 0<\left\vert q\right\vert <1,\ 0\neq z\in\mathbb{C}.
\]
The above $q$-standard notation can be found in \cite{andrew}.

Over 70 years ago, Carlitz extended the classical Bernoulli and Euler numbes
and polynomials and introduced the $q$-Bernoulli and the $q$-Euler numbers and
polynomials (see \cite{calitz1}, \cite{calitz2} and \cite{calitz3} ). There
are numerous recent investigations on this subject by, among many other
authors, Cenki et al. (\cite{cenkci1}, \cite{cenkci2}, \cite{cenkci3}), Choi
et al. (\cite{choi2} and \cite{choi3}), Kim et al. (\cite{kim1}-\cite{kim8}),
Ozden and Simsek \cite{ozden}, Ryoo et al. \cite{tyoo1}, Simsek
(\cite{simsek1}, \cite{simsek2} and \cite{simsek3}), and Luo and Srivastava
\cite{sri1}, Srivastava et al. \cite{sri}.

We first give here the definitions of the $q$-Bernoulli and the $q$-Euler
polynomials of higher order as follows.

\begin{definition}
\label{D:1}Let $q,\alpha\in\mathbb{C},\ 0<\left\vert q\right\vert <1.$ The
$q$-Bernoulli numbers $\mathfrak{B}_{n,q}^{\left(  \alpha\right)  }$ and
polynomials $\mathfrak{B}_{n,q}^{\left(  \alpha\right)  }\left(  x,y\right)  $
in $x,y$ of order $\alpha$ are defined by means of the generating function
functions:%
\begin{align*}
\left(  \frac{t}{e_{q}\left(  t\right)  -1}\right)  ^{\alpha}  &  =\sum
_{n=0}^{\infty}\mathfrak{B}_{n,q}^{\left(  \alpha\right)  }\frac{t^{n}%
}{\left[  n\right]  _{q}!},\ \ \ \left\vert t\right\vert <2\pi,\\
\left(  \frac{t}{e_{q}\left(  t\right)  -1}\right)  ^{\alpha}e_{q}\left(
tx\right)  E_{q}\left(  ty\right)   &  =\sum_{n=0}^{\infty}\mathfrak{B}%
_{n,q}^{\left(  \alpha\right)  }\left(  x,y\right)  \frac{t^{n}}{\left[
n\right]  _{q}!},\ \ \ \left\vert t\right\vert <2\pi.
\end{align*}

\end{definition}

\begin{definition}
\label{D:2}Let $q,\alpha\in\mathbb{C},\ 0<\left\vert q\right\vert <1.$ The
$q$-Euler numbers $\mathfrak{E}_{n,q}^{\left(  \alpha\right)  }$ and
polynomials $\mathfrak{E}_{n,q}^{\left(  \alpha\right)  }\left(  x,y\right)  $
in $x,y$ of order $\alpha$ are defined by means of the generating functions:%
\begin{align*}
\left(  \frac{2}{e_{q}\left(  t\right)  +1}\right)  ^{\alpha}  &  =\sum
_{n=0}^{\infty}\mathfrak{E}_{n,q}^{\left(  \alpha\right)  }\frac{t^{n}%
}{\left[  n\right]  _{q}!},\ \ \ \left\vert t\right\vert <\pi,\\
\left(  \frac{2}{e_{q}\left(  t\right)  +1}\right)  ^{\alpha}e_{q}\left(
tx\right)  E_{q}\left(  ty\right)   &  =\sum_{n=0}^{\infty}\mathfrak{E}%
_{n,q}^{\left(  \alpha\right)  }\left(  x,y\right)  \frac{t^{n}}{\left[
n\right]  _{q}!},\ \ \ \left\vert t\right\vert <\pi.
\end{align*}

\end{definition}

It is obvious that%
\begin{align*}
\mathfrak{B}_{n,q}^{\left(  \alpha\right)  }  &  =\mathfrak{B}_{n,q}^{\left(
\alpha\right)  }\left(  0,0\right)  ,\ \ \ \lim_{q\rightarrow1^{-}%
}\mathfrak{B}_{n,q}^{\left(  \alpha\right)  }\left(  x,y\right)
=B_{n}^{\left(  \alpha\right)  }\left(  x+y\right)  ,\ \ \ \lim_{q\rightarrow
1^{-}}\mathfrak{B}_{n,q}^{\left(  \alpha\right)  }=B_{n}^{\left(
\alpha\right)  },\\
\mathfrak{E}_{n,q}^{\left(  \alpha\right)  }  &  =\mathfrak{E}_{n,q}^{\left(
\alpha\right)  }\left(  0,0\right)  ,\ \ \ \lim_{q\rightarrow1^{-}%
}\mathfrak{E}_{n,q}^{\left(  \alpha\right)  }\left(  x,y\right)
=E_{n}^{\left(  \alpha\right)  }\left(  x+y\right)  ,\ \ \ \lim_{q\rightarrow
1^{-}}\mathfrak{E}_{n,q}^{\left(  \alpha\right)  }=E_{n}^{\left(
\alpha\right)  },\\
\lim_{q\rightarrow1^{-}}\mathfrak{B}_{n,q}^{\left(  \alpha\right)  }\left(
x,0\right)   &  =B_{n}^{\left(  \alpha\right)  }\left(  x\right)
,\ \ \ \lim_{q\rightarrow1^{-}}\mathfrak{B}_{n,q}^{\left(  \alpha\right)
}\left(  0,y\right)  =B_{n}^{\left(  \alpha\right)  }\left(  y\right)
,\ \ \ \\
\lim_{q\rightarrow1^{-}}\mathfrak{E}_{n,q}^{\left(  \alpha\right)  }\left(
x,0\right)   &  =E_{n}^{\left(  \alpha\right)  }\left(  x\right)
,\ \ \ \lim_{q\rightarrow1^{-}}\mathfrak{E}_{n,q}^{\left(  \alpha\right)
}\left(  0,y\right)  =E_{n}^{\left(  \alpha\right)  }\left(  y\right)  .
\end{align*}
Here $B_{n}^{\left(  \alpha\right)  }\left(  x\right)  $ and $E_{n}^{\left(
\alpha\right)  }\left(  x\right)  $ denote the classical Bernoulli and Euler
polynomials of order $\alpha$ which are defined by%
\[
\left(  \frac{t}{e^{t}-1}\right)  ^{\alpha}e^{tx}=\sum_{n=0}^{\infty}%
B_{n}^{\left(  \alpha\right)  }\left(  x\right)  \frac{t^{n}}{\left[
n\right]  _{q}!}\ \ \ \ \ \text{and\ \ \ \ }\left(  \frac{2}{e^{t}+1}\right)
^{\alpha}e^{tx}=\sum_{n=0}^{\infty}E_{n}^{\left(  \alpha\right)  }\left(
x\right)  \frac{t^{n}}{\left[  n\right]  _{q}!}.
\]
In fact Definitions \ref{D:1} and \ref{D:2} define two different type
$\mathfrak{B}_{n,q}^{\left(  \alpha\right)  }\left(  x,0\right)  $ and
$\mathfrak{B}_{n,q}^{\left(  \alpha\right)  }\left(  0,y\right)  $ of the
$q$-Bernoulli polynomials and two different type $\mathfrak{E}_{n,q}^{\left(
\alpha\right)  }\left(  x,0\right)  $ and $\mathfrak{E}_{n,q}^{\left(
\alpha\right)  }\left(  0,y\right)  $ of the $q$-Euler polynomials. Both
polynomials $\mathfrak{B}_{n,q}^{\left(  \alpha\right)  }\left(  x,0\right)  $
and $\mathfrak{B}_{n,q}^{\left(  \alpha\right)  }\left(  0,y\right)  $
($\mathfrak{E}_{n,q}^{\left(  \alpha\right)  }\left(  x,0\right)  $ and
$\mathfrak{E}_{n,q}^{\left(  \alpha\right)  }\left(  0,y\right)  $) coincide
with the classical highe order Bernoulli polynomilas (Euler polynomilas) in
the limiting case $q\rightarrow1^{-}$.

For the $q$-Bernoulli numbers $\mathfrak{B}_{n,q}$, the $q$-Euler numbers
$\mathfrak{E}_{n,q}$ of order $n$, we have%
\[
\mathfrak{B}_{n,q}=\mathfrak{B}_{n,q}\left(  0,0\right)  =\mathfrak{B}%
_{n,q}^{\left(  1\right)  }\left(  0,0\right)  ,\ \ \ \ \ \mathfrak{E}%
_{n,q}=\mathfrak{E}_{n,q}\left(  0,0\right)  =\mathfrak{E}_{n,q}^{\left(
1\right)  }\left(  0,0\right)  ,
\]
respectively. Note that the $q$-Bernoulli numbers $\mathfrak{B}_{n,q}$ are
defined and studied in \cite{manna}.

The aim of the present paper is to obtain some results for the above defined
$q$-Bernoulli and $q$-Euler polynomials. In this paper the $q$-analogues of
well-known results, for example, Srivastava and Pint\'{e}r \cite{pinter},
Cheon \cite{cheon}, etc., will be given. Also the formulas involving the
$q$-Stirling numbers of the second kind, $q$-Bernoulli polynomials and
Phillips $q$-Bernstein polynomials are derived.

\section{Preliminaries and Lemmas}

In this section we shall provide some basic formulas for the $q$-Bernoulli and
$q$-Euler polynomials in order to obtain the main results of this paper in the
next section. The following result is $q$-analogue of the addition theorem for
the classical Bernoulli and Euler polynomials.

\begin{lemma}
\label{L:1}\emph{(Addition Theorems)} For all $x,y\in\mathbb{C}$ we have
\begin{align}
\mathfrak{B}_{n,q}^{\left(  \alpha\right)  }\left(  x,y\right)   &  =%
%TCIMACRO{\dsum \limits_{k=0}^{n}}%
%BeginExpansion
{\displaystyle\sum\limits_{k=0}^{n}}
%EndExpansion
\left[
\begin{array}
[c]{c}%
n\\
k
\end{array}
\right]  _{q}\mathfrak{B}_{k,q}^{\left(  \alpha\right)  }\left(  x+y\right)
_{q}^{n-k},\ \ \ \mathfrak{E}_{n,q}^{\left(  \alpha\right)  }\left(
x,y\right)  =%
%TCIMACRO{\dsum \limits_{k=0}^{n}}%
%BeginExpansion
{\displaystyle\sum\limits_{k=0}^{n}}
%EndExpansion
\left[
\begin{array}
[c]{c}%
n\\
k
\end{array}
\right]  _{q}\ \mathfrak{E}_{k,q}^{\left(  \alpha\right)  }\left(  x+y\right)
_{q}^{n-k},\nonumber\\
\mathfrak{B}_{n,q}^{\left(  \alpha\right)  }\left(  x,y\right)   &
=\sum_{k=0}^{n}\left[
\begin{array}
[c]{c}%
n\\
k
\end{array}
\right]  _{q}q^{\left(  n-k\right)  \left(  n-k-1\right)  /2}\mathfrak{B}%
_{k,q}^{\left(  \alpha\right)  }\left(  x,0\right)  y^{n-k}=\sum_{k=0}%
^{n}\left[
\begin{array}
[c]{c}%
n\\
k
\end{array}
\right]  _{q}\mathfrak{B}_{k,q}^{\left(  \alpha\right)  }\left(  0,y\right)
x^{n-k},\label{be1}\\
\mathfrak{E}_{n,q}^{\left(  \alpha\right)  }\left(  x,y\right)   &
=\sum_{k=0}^{n}\left[
\begin{array}
[c]{c}%
n\\
k
\end{array}
\right]  _{q}q^{\left(  n-k\right)  \left(  n-k-1\right)  /2}\mathfrak{E}%
_{k,q}^{\left(  \alpha\right)  }\left(  x,0\right)  y^{n-k}=\sum_{k=0}%
^{n}\left[
\begin{array}
[c]{c}%
n\\
k
\end{array}
\right]  _{q}\mathfrak{E}_{k,q}^{\left(  \alpha\right)  }\left(  0,y\right)
x^{n-k}. \label{be2}%
\end{align}

\end{lemma}

In particular, setting $x=0$ and $y=0$ in (\ref{be1}) and (\ref{be2}), we get
the following formulas for $q$-Bernoulli and $q$-Euler polynomials,
respectively.%
\begin{align}
\mathfrak{B}_{n,q}^{\left(  \alpha\right)  }\left(  x,0\right)   &
=\sum_{k=0}^{n}\left[
\begin{array}
[c]{c}%
n\\
k
\end{array}
\right]  _{q}\mathfrak{B}_{k,q}^{\left(  \alpha\right)  }x^{n-k}%
,\ \ \ \mathfrak{B}_{n,q}^{\left(  \alpha\right)  }\left(  0,y\right)
=\sum_{k=0}^{n}\left[
\begin{array}
[c]{c}%
n\\
k
\end{array}
\right]  _{q}q^{\left(  n-\not k  \right)  \left(  n-k-1\right)
/2}\mathfrak{B}_{k,q}^{\left(  \alpha\right)  }y^{n-k},\label{be7}\\
\mathfrak{E}_{n,q}^{\left(  \alpha\right)  }\left(  x,0\right)   &
=\sum_{k=0}^{n}\left[
\begin{array}
[c]{c}%
n\\
k
\end{array}
\right]  _{q}\mathfrak{E}_{k,q}^{\left(  \alpha\right)  }x^{n-k}%
,\ \ \ \mathfrak{E}_{n,q}^{\left(  \alpha\right)  }\left(  0,y\right)
=\sum_{k=0}^{n}\left[
\begin{array}
[c]{c}%
n\\
k
\end{array}
\right]  _{q}q^{\left(  n-\not k  \right)  \left(  n-k-1\right)
/2}\mathfrak{E}_{k,q}^{\left(  \alpha\right)  }y^{n-k}. \label{be8}%
\end{align}
Setting $y=1$ and $x=1$ in (\ref{be1}) and (\ref{be2}), we get%
\begin{align}
\mathfrak{B}_{n,q}^{\left(  \alpha\right)  }\left(  x,1\right)   &
=\sum_{k=0}^{n}\left[
\begin{array}
[c]{c}%
n\\
k
\end{array}
\right]  _{q}q^{\left(  n-k\right)  \left(  n-k-1\right)  /2}\mathfrak{B}%
_{k,q}^{\left(  \alpha\right)  }\left(  x,0\right)  ,\ \ \ \mathfrak{B}%
_{n,q}^{\left(  \alpha\right)  }\left(  1,y\right)  =\sum_{k=0}^{n}\left[
\begin{array}
[c]{c}%
n\\
k
\end{array}
\right]  _{q}\mathfrak{B}_{k,q}^{\left(  \alpha\right)  }\left(  0,y\right)
,\label{be3}\\
\mathfrak{E}_{n,q}^{\left(  \alpha\right)  }\left(  x,1\right)   &
=\sum_{k=0}^{n}\left[
\begin{array}
[c]{c}%
n\\
k
\end{array}
\right]  _{q}q^{\left(  n-k\right)  \left(  n-k-1\right)  /2}\mathfrak{E}%
_{k,q}^{\left(  \alpha\right)  }\left(  x,0\right)  ,\ \ \ \mathfrak{E}%
_{n,q}^{\left(  \alpha\right)  }\left(  1,y\right)  =\sum_{k=0}^{n}\left[
\begin{array}
[c]{c}%
n\\
k
\end{array}
\right]  _{q}\mathfrak{E}_{k,q}^{\left(  \alpha\right)  }\left(  0,y\right)  .
\label{be4}%
\end{align}
Clearly (\ref{be3}) and (\ref{be4}) are $q$-analogues of%
\[
B_{n}^{\left(  \alpha\right)  }\left(  x+1\right)  =\sum_{k=0}^{n}\left(
\begin{array}
[c]{c}%
n\\
k
\end{array}
\right)  B_{k}^{\left(  \alpha\right)  }\left(  x\right)  ,\ \ \ E_{n}%
^{\left(  \alpha\right)  }\left(  x+1\right)  =\sum_{k=0}^{n}\left(
\begin{array}
[c]{c}%
n\\
k
\end{array}
\right)  E_{k}^{\left(  \alpha\right)  }\left(  x\right)  ,
\]
respectively.

\begin{lemma}
\label{L:2}We have%
\begin{align*}
D_{q,x}\mathfrak{B}_{n,q}^{\left(  \alpha\right)  }\left(  x,y\right)   &
=\left[  n\right]  _{q}\mathfrak{B}_{n-1,q}^{\left(  \alpha\right)  }\left(
x,y\right)  ,\ \ \ D_{q,y}\mathfrak{B}_{n,q}^{\left(  \alpha\right)  }\left(
x,y\right)  =\left[  n\right]  _{q}\mathfrak{B}_{n-1,q}^{\left(
\alpha\right)  }\left(  x,qy\right)  ,\\
D_{q,x}\mathfrak{E}_{n,q}^{\left(  \alpha\right)  }\left(  x,y\right)   &
=\left[  n\right]  _{q}\mathfrak{E}_{n-1,q}^{\left(  \alpha\right)  }\left(
x,y\right)  ,\ \ \ D_{q,y}\mathfrak{E}_{n,q}^{\left(  \alpha\right)  }\left(
x,y\right)  =\left[  n\right]  _{q}\ \mathfrak{E}_{n-1,q}^{\left(
\alpha\right)  }\left(  x,qy\right)  .
\end{align*}

\end{lemma}

\begin{lemma}
\label{L:3}\emph{(Difference Equations)} We have%
\begin{align}
\mathfrak{B}_{n,q}^{\left(  \alpha\right)  }\left(  1,y\right)  -\mathfrak{B}%
_{n,q}^{\left(  \alpha\right)  }\left(  0,y\right)   &  =\left[  n\right]
_{q}\mathfrak{B}_{n-1,q}^{\left(  \alpha-1\right)  }\left(  0,y\right)
,\label{be5}\\
\mathfrak{E}_{n,q}^{\left(  \alpha\right)  }\left(  1,y\right)  +\mathfrak{E}%
_{n,q}^{\left(  \alpha\right)  }\left(  0,y\right)   &  =2\mathfrak{E}%
_{n,q}^{\left(  \alpha-1\right)  }\left(  0,y\right)  ,\label{be6}\\
\mathfrak{B}_{n,q}^{\left(  \alpha\right)  }\left(  x,0\right)  -\mathfrak{B}%
_{n,q}^{\left(  \alpha\right)  }\left(  x,-1\right)   &  =\left[  n\right]
_{q}\mathfrak{B}_{n-1,q}^{\left(  \alpha-1\right)  }\left(  x,-1\right)
,\nonumber\\
\mathfrak{E}_{n,q}^{\left(  \alpha\right)  }\left(  x,0\right)  +\mathfrak{E}%
_{n,q}^{\left(  \alpha\right)  }\left(  x,-1\right)   &  =2\mathfrak{E}%
_{n,q}^{\left(  \alpha-1\right)  }\left(  x,-1\right)  .\nonumber
\end{align}

\end{lemma}

From (\ref{be5}) and (\ref{be7}), (\ref{be6}) and (\ref{be8}) we obtain the
following formulas.

\begin{lemma}
\label{L:4}We have
\begin{align}
\mathfrak{B}_{n-1,q}^{\left(  \alpha-1\right)  }\left(  0,y\right)   &
=\frac{1}{\left[  n+1\right]  _{q}}\sum_{k=0}^{n}\left[
\begin{array}
[c]{c}%
n+1\\
k
\end{array}
\right]  _{q}\mathfrak{B}_{k,q}^{\left(  \alpha\right)  }\left(  0,y\right)
,\label{be9}\\
\mathfrak{E}_{n,q}^{\left(  \alpha-1\right)  }\left(  0,y\right)   &
=\frac{1}{2}\left[  \sum_{k=0}^{n}\left[
\begin{array}
[c]{c}%
n\\
k
\end{array}
\right]  _{q}\mathfrak{E}_{k,q}^{\left(  \alpha\right)  }\left(  0,y\right)
+\mathfrak{E}_{n,q}^{\left(  \alpha\right)  }\left(  0,y\right)  \right]  .
\label{be10}%
\end{align}

\end{lemma}

Putting $\alpha=1$ in (\ref{be9}) and (\ref{be10}), and noting that%
\[
\mathfrak{B}_{n,q}^{\left(  0\right)  }\left(  0,y\right)  =\mathfrak{E}%
_{n,q}^{\left(  0\right)  }\left(  0,y\right)  =q^{n\left(  n-1\right)
/2}y^{n},
\]
we arrive at the following expansions:
\begin{align*}
y^{n}  &  =\frac{1}{q^{n\left(  n-1\right)  /2}\left[  n+1\right]  _{q}}%
%TCIMACRO{\dsum \limits_{k=0}^{n}}%
%BeginExpansion
{\displaystyle\sum\limits_{k=0}^{n}}
%EndExpansion
\left[
\begin{array}
[c]{c}%
n+1\\
k
\end{array}
\right]  _{q}\mathfrak{B}_{k,q}\left(  0,y\right)  ,\\
y^{n}  &  =\frac{1}{2q^{n\left(  n-1\right)  /2}}\left[
%TCIMACRO{\dsum \limits_{k=0}^{n}}%
%BeginExpansion
{\displaystyle\sum\limits_{k=0}^{n}}
%EndExpansion
\left[
\begin{array}
[c]{c}%
n\\
k
\end{array}
\right]  _{q}\mathfrak{E}_{k,q}\left(  0,y\right)  +\mathfrak{E}_{n,q}\left(
0,y\right)  \right]  ,
\end{align*}
which are $q$-analoques of the following familiar expansions%
\begin{equation}
y^{n}=\frac{1}{n+1}%
%TCIMACRO{\dsum \limits_{k=0}^{n}}%
%BeginExpansion
{\displaystyle\sum\limits_{k=0}^{n}}
%EndExpansion
\left(
\begin{array}
[c]{c}%
n+1\\
k
\end{array}
\right)  B_{k}\left(  y\right)  ,\ \ \ y^{n}=\frac{1}{2}\left[
%TCIMACRO{\dsum \limits_{k=0}^{n}}%
%BeginExpansion
{\displaystyle\sum\limits_{k=0}^{n}}
%EndExpansion
\left(
\begin{array}
[c]{c}%
n\\
k
\end{array}
\right)  E_{k}\left(  y\right)  +E_{n}\left(  y\right)  \right]  , \label{cl1}%
\end{equation}
respectively.

\begin{lemma}
\label{L:5}\emph{(Recurrence Relationships) }The polynomials $\mathfrak{B}%
_{n,q}^{\left(  \alpha\right)  }\left(  x,0\right)  $ and $\mathfrak{E}%
_{n,q}^{\left(  \alpha\right)  }\left(  x,0\right)  $ satisfy the following
difference relationships:%
\begin{align}%
%TCIMACRO{\dsum \limits_{j=0}^{k}}%
%BeginExpansion
{\displaystyle\sum\limits_{j=0}^{k}}
%EndExpansion
\left[
\begin{array}
[c]{c}%
k\\
j
\end{array}
\right]  _{q}m^{j}\mathfrak{B}_{j,q}^{\left(  \alpha\right)  }\left(
x,0\right)  -%
%TCIMACRO{\dsum \limits_{j=0}^{k}}%
%BeginExpansion
{\displaystyle\sum\limits_{j=0}^{k}}
%EndExpansion
\left[
\begin{array}
[c]{c}%
k\\
j
\end{array}
\right]  _{q}m^{j}\mathfrak{B}_{j,q}^{\left(  \alpha\right)  }\left(
x,-1\right)   &  =\left[  k\right]  _{q}%
%TCIMACRO{\dsum \limits_{j=0}^{k-1}}%
%BeginExpansion
{\displaystyle\sum\limits_{j=0}^{k-1}}
%EndExpansion
\left[
\begin{array}
[c]{c}%
k-1\\
j
\end{array}
\right]  _{q}m^{j+1}\mathfrak{B}_{j,q}^{\left(  \alpha-1\right)  }\left(
x,-1\right)  ,\label{be11}\\
\mathfrak{B}_{k,q}^{\left(  \alpha\right)  }\left(  \frac{1}{m},y\right)  -%
%TCIMACRO{\dsum \limits_{j=0}^{k}}%
%BeginExpansion
{\displaystyle\sum\limits_{j=0}^{k}}
%EndExpansion
\left[
\begin{array}
[c]{c}%
k\\
j
\end{array}
\right]  _{q}\left(  \frac{1}{m}-1\right)  _{q}^{k-j}\mathfrak{B}%
_{j,q}^{\left(  \alpha\right)  }\left(  0,y\right)   &  =\left[  k\right]
_{q}%
%TCIMACRO{\dsum \limits_{j=0}^{k-1}}%
%BeginExpansion
{\displaystyle\sum\limits_{j=0}^{k-1}}
%EndExpansion
\left[
\begin{array}
[c]{c}%
k-1\\
j
\end{array}
\right]  _{q}\left(  \frac{1}{m}-1\right)  _{q}^{k-j-1}\mathfrak{B}%
_{j,q}^{\left(  \alpha-1\right)  }\left(  0,y\right)  ,\label{be11-1}\\%
%TCIMACRO{\dsum \limits_{j=0}^{k}}%
%BeginExpansion
{\displaystyle\sum\limits_{j=0}^{k}}
%EndExpansion
\left[
\begin{array}
[c]{c}%
k\\
j
\end{array}
\right]  _{q}m^{j}\mathfrak{E}_{j,q}^{\left(  \alpha\right)  }\left(
x,0\right)  +%
%TCIMACRO{\dsum \limits_{j=0}^{k}}%
%BeginExpansion
{\displaystyle\sum\limits_{j=0}^{k}}
%EndExpansion
\left[
\begin{array}
[c]{c}%
k\\
j
\end{array}
\right]  _{q}m^{j}\mathfrak{E}_{j,q}^{\left(  \alpha\right)  }\left(
x,-1\right)   &  =2%
%TCIMACRO{\dsum \limits_{j=0}^{k}}%
%BeginExpansion
{\displaystyle\sum\limits_{j=0}^{k}}
%EndExpansion
\left[
\begin{array}
[c]{c}%
k\\
j
\end{array}
\right]  _{q}m^{j}\mathfrak{E}_{j,q}^{\left(  \alpha-1\right)  }\left(
x,-1\right)  ,\label{be12}\\
\mathfrak{E}_{k,q}^{\left(  \alpha\right)  }\left(  \frac{1}{m},y\right)  +%
%TCIMACRO{\dsum \limits_{j=0}^{k}}%
%BeginExpansion
{\displaystyle\sum\limits_{j=0}^{k}}
%EndExpansion
\left[
\begin{array}
[c]{c}%
k\\
j
\end{array}
\right]  _{q}\left(  \frac{1}{m}-1\right)  _{q}^{k-j}\mathfrak{E}%
_{j,q}^{\left(  \alpha\right)  }\left(  0,y\right)   &  =2%
%TCIMACRO{\dsum \limits_{j=0}^{k}}%
%BeginExpansion
{\displaystyle\sum\limits_{j=0}^{k}}
%EndExpansion
\left[
\begin{array}
[c]{c}%
k\\
j
\end{array}
\right]  _{q}\left(  \frac{1}{m}-1\right)  _{q}^{k-j}\mathfrak{E}%
_{j,q}^{\left(  \alpha-1\right)  }\left(  0,y\right)  . \label{be12-1}%
\end{align}

\end{lemma}

\section{Explicit relationship between the $q$-Bernoulli and $q$-Euler
polynomials}

In this section we shall investigate some explicit relationships between the
$q$-Bernoulli and $q$-Euler polynomials. Here some $q$-analogues of known
results will be given. We also obtain new formulas and their some special
cases below. These formulas are some extensions of the formulas of Srivastava
and \'{A}. Pint\'{e}r, Cheon and others.

We present natural $q$-extensions of th main results of the papers
\cite{pinter}, \cite{luo2}, see Theorems \ref{S-P1} and \ref{S-P2}.

\begin{theorem}
\label{S-P1}For $n\in\mathbb{N}_{0}$, the following relationship%
\begin{align*}
\mathfrak{B}_{n,q}^{\left(  \alpha\right)  }\left(  x,y\right)   &  =\frac
{1}{2m^{n}}\sum_{k=0}^{n}\left[
\begin{array}
[c]{c}%
n\\
k
\end{array}
\right]  _{q}\left[  m^{k}\mathfrak{B}_{k,q}^{\left(  \alpha\right)  }\left(
x,0\right)  +%
%TCIMACRO{\dsum \limits_{j=0}^{k}}%
%BeginExpansion
{\displaystyle\sum\limits_{j=0}^{k}}
%EndExpansion
\left[
\begin{array}
[c]{c}%
k\\
j
\end{array}
\right]  _{q}m^{j}\mathfrak{B}_{j,q}^{\left(  \alpha\right)  }\left(
x,-1\right)  \right. \\
&  \left.  +\left[  k\right]  _{q}%
%TCIMACRO{\dsum \limits_{j=0}^{k-1}}%
%BeginExpansion
{\displaystyle\sum\limits_{j=0}^{k-1}}
%EndExpansion
\left[
\begin{array}
[c]{c}%
k-1\\
j
\end{array}
\right]  _{q}m^{j+1}\mathfrak{B}_{j,q}^{\left(  \alpha-1\right)  }\left(
x,-1\right)  \right]  \mathfrak{E}_{n-k,q}\left(  0,my\right)  ,\\
\mathfrak{B}_{n,q}^{\left(  \alpha\right)  }\left(  x,y\right)   &  =\frac
{1}{2m^{n}}%
%TCIMACRO{\dsum \limits_{k=0}^{n}}%
%BeginExpansion
{\displaystyle\sum\limits_{k=0}^{n}}
%EndExpansion
\left[
\begin{array}
[c]{c}%
n\\
j
\end{array}
\right]  _{q}m^{k}\left[  \mathfrak{B}_{k,q}^{\left(  \alpha\right)  }\left(
0,y\right)  +%
%TCIMACRO{\dsum \limits_{j=0}^{k}}%
%BeginExpansion
{\displaystyle\sum\limits_{j=0}^{k}}
%EndExpansion
\left[
\begin{array}
[c]{c}%
k\\
j
\end{array}
\right]  _{q}\left(  \frac{1}{m}-1\right)  _{q}^{k-j}\mathfrak{B}%
_{j,q}^{\left(  \alpha\right)  }\left(  0,y\right)  \right. \\
&  \left.  +\left[  k\right]  _{q}%
%TCIMACRO{\dsum \limits_{j=0}^{k-1}}%
%BeginExpansion
{\displaystyle\sum\limits_{j=0}^{k-1}}
%EndExpansion
\left[
\begin{array}
[c]{c}%
k-1\\
j
\end{array}
\right]  _{q}\left(  \frac{1}{m}-1\right)  _{q}^{k-1-j}\mathfrak{B}%
_{j,q}^{\left(  \alpha-1\right)  }\left(  0,y\right)  \right]  \mathfrak{E}%
_{n-k,q}\left(  mx,0\right)
\end{align*}
holds true between the $q$-Bernoulli polynomials and $q$-Euler polynomials.
\end{theorem}

\begin{proof}
Using the following identity%
\[
\left(  \frac{t}{e_{q}\left(  t\right)  -1}\right)  ^{\alpha}e_{q}\left(
tx\right)  E_{q}\left(  ty\right)  =\frac{2}{e_{q}\left(  \frac{t}{m}\right)
+1}\cdot E_{q}\left(  \frac{t}{m}my\right)  \cdot\frac{e_{q}\left(  \frac
{t}{m}\right)  +1}{2}\cdot\left(  \frac{t}{e_{q}\left(  t\right)  -1}\right)
^{\alpha}e_{q}\left(  tx\right)
\]
we have%
\begin{align*}%
%TCIMACRO{\dsum \limits_{n=0}^{\infty}}%
%BeginExpansion
{\displaystyle\sum\limits_{n=0}^{\infty}}
%EndExpansion
\mathfrak{B}_{n,q}^{\left(  \alpha\right)  }\left(  x,y\right)  \frac{t^{n}%
}{\left[  n\right]  _{q}!}  &  =\frac{1}{2}%
%TCIMACRO{\dsum \limits_{n=0}^{\infty}}%
%BeginExpansion
{\displaystyle\sum\limits_{n=0}^{\infty}}
%EndExpansion
\mathfrak{E}_{n,q}\left(  0,my\right)  \frac{t^{n}}{m^{n}\left[  n\right]
_{q}!}%
%TCIMACRO{\dsum \limits_{n=0}^{\infty}}%
%BeginExpansion
{\displaystyle\sum\limits_{n=0}^{\infty}}
%EndExpansion
\frac{t^{n}}{m^{n}\left[  n\right]  _{q}!}%
%TCIMACRO{\dsum \limits_{n=0}^{\infty}}%
%BeginExpansion
{\displaystyle\sum\limits_{n=0}^{\infty}}
%EndExpansion
\mathfrak{B}_{n,q}^{\left(  \alpha\right)  }\left(  x,0\right)  \frac{t^{n}%
}{\left[  n\right]  _{q}!}\\
&  +\frac{1}{2}%
%TCIMACRO{\dsum \limits_{n=0}^{\infty}}%
%BeginExpansion
{\displaystyle\sum\limits_{n=0}^{\infty}}
%EndExpansion
\mathfrak{E}_{n,q}\left(  0,my\right)  \frac{t^{n}}{m^{n}\left[  n\right]
_{q}!}%
%TCIMACRO{\dsum \limits_{n=0}^{\infty}}%
%BeginExpansion
{\displaystyle\sum\limits_{n=0}^{\infty}}
%EndExpansion
\mathfrak{B}_{n,q}^{\left(  \alpha\right)  }\left(  x,0\right)  \frac{t^{n}%
}{\left[  n\right]  _{q}!}\\
&  =:I_{1}+I_{2}.
\end{align*}
It is clear that%
\[
I_{2}=\frac{1}{2}%
%TCIMACRO{\dsum \limits_{n=0}^{\infty}}%
%BeginExpansion
{\displaystyle\sum\limits_{n=0}^{\infty}}
%EndExpansion
\mathfrak{E}_{n,q}\left(  0,my\right)  \frac{t^{n}}{m^{n}\left[  n\right]
_{q}!}%
%TCIMACRO{\dsum \limits_{n=0}^{\infty}}%
%BeginExpansion
{\displaystyle\sum\limits_{n=0}^{\infty}}
%EndExpansion
\mathfrak{B}_{n,q}^{\left(  \alpha\right)  }\left(  x,0\right)  \frac{t^{n}%
}{\left[  n\right]  _{q}!}=\frac{1}{2}%
%TCIMACRO{\dsum \limits_{n=0}^{\infty}}%
%BeginExpansion
{\displaystyle\sum\limits_{n=0}^{\infty}}
%EndExpansion%
%TCIMACRO{\dsum \limits_{k=0}^{n}}%
%BeginExpansion
{\displaystyle\sum\limits_{k=0}^{n}}
%EndExpansion
\left[
\begin{array}
[c]{c}%
n\\
j
\end{array}
\right]  _{q}m^{k-n}\mathfrak{B}_{k,q}^{\left(  \alpha\right)  }\left(
x,0\right)  \mathfrak{E}_{n-k,q}\left(  0,my\right)  \frac{t^{n}}{\left[
n\right]  _{q}!}.
\]
On the other hand%
\begin{align*}
I_{1}  &  =\frac{1}{2}%
%TCIMACRO{\dsum \limits_{n=0}^{\infty}}%
%BeginExpansion
{\displaystyle\sum\limits_{n=0}^{\infty}}
%EndExpansion
\mathfrak{B}_{n,q}^{\left(  \alpha\right)  }\left(  x,0\right)  \frac{t^{n}%
}{\left[  n\right]  _{q}!}%
%TCIMACRO{\dsum \limits_{n=0}^{\infty}}%
%BeginExpansion
{\displaystyle\sum\limits_{n=0}^{\infty}}
%EndExpansion%
%TCIMACRO{\dsum \limits_{j=0}^{n}}%
%BeginExpansion
{\displaystyle\sum\limits_{j=0}^{n}}
%EndExpansion
\left[
\begin{array}
[c]{c}%
n\\
j
\end{array}
\right]  _{q}m^{-n}\mathfrak{E}_{j,q}\left(  0,my\right)  \frac{t^{n}}{\left[
n\right]  _{q}!}\\
&  =\frac{1}{2}%
%TCIMACRO{\dsum \limits_{n=0}^{\infty}}%
%BeginExpansion
{\displaystyle\sum\limits_{n=0}^{\infty}}
%EndExpansion%
%TCIMACRO{\dsum \limits_{k=0}^{n}}%
%BeginExpansion
{\displaystyle\sum\limits_{k=0}^{n}}
%EndExpansion
\left[
\begin{array}
[c]{c}%
n\\
k
\end{array}
\right]  _{q}\mathfrak{B}_{k,q}^{\left(  \alpha\right)  }\left(  x,0\right)
%TCIMACRO{\dsum \limits_{j=0}^{n-k}}%
%BeginExpansion
{\displaystyle\sum\limits_{j=0}^{n-k}}
%EndExpansion
\left[
\begin{array}
[c]{c}%
n-k\\
j
\end{array}
\right]  _{q}m^{k-n}\mathfrak{E}_{j,q}\left(  0,my\right)  \frac{t^{n}%
}{\left[  n\right]  _{q}!}\\
&  =\frac{1}{2}%
%TCIMACRO{\dsum \limits_{n=0}^{\infty}}%
%BeginExpansion
{\displaystyle\sum\limits_{n=0}^{\infty}}
%EndExpansion
m^{-n}%
%TCIMACRO{\dsum \limits_{j=0}^{n}}%
%BeginExpansion
{\displaystyle\sum\limits_{j=0}^{n}}
%EndExpansion
\left[
\begin{array}
[c]{c}%
n\\
j
\end{array}
\right]  _{q}\mathfrak{E}_{j,q}\left(  0,my\right)
%TCIMACRO{\dsum \limits_{k=0}^{n-j}}%
%BeginExpansion
{\displaystyle\sum\limits_{k=0}^{n-j}}
%EndExpansion
\left[
\begin{array}
[c]{c}%
n-j\\
k
\end{array}
\right]  _{q}m^{k}\mathfrak{B}_{k,q}^{\left(  \alpha\right)  }\left(
x,0\right)  \frac{t^{n}}{\left[  n\right]  _{q}!}.
\end{align*}
Therefore%
\[%
%TCIMACRO{\dsum \limits_{n=0}^{\infty}}%
%BeginExpansion
{\displaystyle\sum\limits_{n=0}^{\infty}}
%EndExpansion
\mathfrak{B}_{n,q}^{\left(  \alpha\right)  }\left(  x,y\right)  \frac{t^{n}%
}{\left[  n\right]  _{q}!}=\frac{1}{2}%
%TCIMACRO{\dsum \limits_{n=0}^{\infty}}%
%BeginExpansion
{\displaystyle\sum\limits_{n=0}^{\infty}}
%EndExpansion%
%TCIMACRO{\dsum \limits_{k=0}^{n}}%
%BeginExpansion
{\displaystyle\sum\limits_{k=0}^{n}}
%EndExpansion
\left[
\begin{array}
[c]{c}%
n\\
j
\end{array}
\right]  _{q}m^{k-n}\left[  \mathfrak{B}_{k,q}^{\left(  \alpha\right)
}\left(  x,0\right)  +m^{-k}%
%TCIMACRO{\dsum \limits_{j=0}^{k}}%
%BeginExpansion
{\displaystyle\sum\limits_{j=0}^{k}}
%EndExpansion
\left[
\begin{array}
[c]{c}%
k\\
j
\end{array}
\right]  _{q}m^{j}\mathfrak{B}_{j,q}^{\left(  \alpha\right)  }\left(
x,0\right)  \right]  \mathfrak{E}_{n-k,q}\left(  0,my\right)  \frac{t^{n}%
}{\left[  n\right]  _{q}!}.
\]
It remains to use the formula (\ref{be11}).
\end{proof}

Next we discuss some special cases of Theorem \ref{S-P1}.

\begin{corollary}
\label{C:1}For $n\in\mathbb{N}_{0}$, $m\in\mathbb{N}$ the following
relationship%
\begin{align*}
\mathfrak{B}_{n,q}\left(  x,y\right)   &  =\frac{1}{2m^{n}}\sum_{k=0}%
^{n}\left[
\begin{array}
[c]{c}%
n\\
k
\end{array}
\right]  _{q}\left[  m^{k}\mathfrak{B}_{k,q}\left(  x,0\right)  +%
%TCIMACRO{\dsum \limits_{j=0}^{k}}%
%BeginExpansion
{\displaystyle\sum\limits_{j=0}^{k}}
%EndExpansion
\left[
\begin{array}
[c]{c}%
k\\
j
\end{array}
\right]  _{q}m^{j}\mathfrak{B}_{j,q}\left(  x,-1\right)  \right. \\
&  +\left.  \left[  k\right]  _{q}%
%TCIMACRO{\dsum \limits_{j=0}^{k-1}}%
%BeginExpansion
{\displaystyle\sum\limits_{j=0}^{k-1}}
%EndExpansion
\left[
\begin{array}
[c]{c}%
k-1\\
j
\end{array}
\right]  _{q}m^{j+1}\left(  x-1\right)  _{q}^{j}\right]  \mathfrak{E}%
_{n-k,q}\left(  0,my\right)  ,\\
\mathfrak{B}_{n,q}\left(  x,y\right)   &  =\frac{1}{2m^{n}}%
%TCIMACRO{\dsum \limits_{k=0}^{n}}%
%BeginExpansion
{\displaystyle\sum\limits_{k=0}^{n}}
%EndExpansion
\left[
\begin{array}
[c]{c}%
n\\
j
\end{array}
\right]  _{q}m^{k}\left[  m^{k}\mathfrak{B}_{k,q}\left(  0,y\right)  +%
%TCIMACRO{\dsum \limits_{j=0}^{k}}%
%BeginExpansion
{\displaystyle\sum\limits_{j=0}^{k}}
%EndExpansion
\left[
\begin{array}
[c]{c}%
k\\
j
\end{array}
\right]  _{q}\left(  \frac{1}{m}-1\right)  _{q}^{k-j}\mathfrak{B}_{j,q}\left(
0,y\right)  \right. \\
&  \left.  +\left[  k\right]  _{q}%
%TCIMACRO{\dsum \limits_{j=0}^{k-1}}%
%BeginExpansion
{\displaystyle\sum\limits_{j=0}^{k-1}}
%EndExpansion
\left[
\begin{array}
[c]{c}%
k-1\\
j
\end{array}
\right]  _{q}q^{\frac{1}{2}j\left(  j-1\right)  }\left(  \frac{1}{m}-1\right)
_{q}^{k-1-j}y^{j}\right]  \mathfrak{E}_{n-k,q}\left(  mx,0\right)
\end{align*}
holds true between the $q$-Bernoulli polynomials and $q$-Euler polynomials.
\end{corollary}

\begin{corollary}
\label{C:2}\cite{luo2} For $n\in\mathbb{N}_{0}$, $m\in\mathbb{N}$ the
following relationship holds true.%
\begin{align*}
B_{n}\left(  x+y\right)   &  =%
%TCIMACRO{\dsum \limits_{k=0}^{n}}%
%BeginExpansion
{\displaystyle\sum\limits_{k=0}^{n}}
%EndExpansion
\left(
\begin{array}
[c]{c}%
n\\
k
\end{array}
\right)  \left(  B_{k}\left(  y\right)  +\frac{k}{2}y^{k-1}\right)
E_{n-k}\left(  x\right)  ,\\
B_{n}\left(  x+y\right)   &  =\frac{1}{2m^{n}}\sum_{k=0}^{n}\left(
\begin{array}
[c]{c}%
n\\
k
\end{array}
\right)  \left[  m^{k}B_{k}\left(  x\right)  +m^{k}B_{k}\left(  x-1+\frac
{1}{m}\right)  +km\left(  1+m\left(  x-1\right)  \right)  ^{k-1}\right]
E_{n-k,q}\left(  my\right)  .
\end{align*}

\end{corollary}

\begin{corollary}
\label{C:3}For $n\in\mathbb{N}_{0}$ the following relationship holds true.%
\begin{equation}
\mathfrak{B}_{n,q}\left(  x,y\right)  =%
%TCIMACRO{\dsum \limits_{k=0}^{n}}%
%BeginExpansion
{\displaystyle\sum\limits_{k=0}^{n}}
%EndExpansion
\left[
\begin{array}
[c]{c}%
n\\
k
\end{array}
\right]  _{q}\left(  \mathfrak{B}_{k,q}\left(  0,y\right)  +q^{\frac{1}%
{2}\left(  k-1\right)  \left(  k-2\right)  }\frac{\left[  k\right]  _{q}}%
{2}y^{k-1}\right)  \mathfrak{E}_{n-k,q}\left(  x,0\right)  . \label{cw1}%
\end{equation}

\end{corollary}

\begin{corollary}
\label{C:4}For $n\in\mathbb{N}_{0}$ the following relationship holds true.%
\begin{align}
\mathfrak{B}_{n,q}\left(  x,0\right)   &  =%
%TCIMACRO{\dsum \limits_{\substack{k=0\\\left(  k\neq1\right)  }}^{n}}%
%BeginExpansion
{\displaystyle\sum\limits_{\substack{k=0\\\left(  k\neq1\right)  }}^{n}}
%EndExpansion
\left[
\begin{array}
[c]{c}%
n\\
k
\end{array}
\right]  \mathfrak{B}_{k,q}\mathfrak{E}_{n-k,q}\left(  x,0\right)  +\left(
\mathfrak{B}_{1,q}+\frac{1}{2}\right)  \mathfrak{E}_{n-1,q}\left(  x,0\right)
,\label{cw2}\\
\mathfrak{B}_{n,q}\left(  0,y\right)   &  =\sum_{\substack{k=0\\\left(
k\neq1\right)  }}^{n}\left[
\begin{array}
[c]{c}%
n\\
k
\end{array}
\right]  _{q}\mathfrak{B}_{k,q}\mathfrak{E}_{n-k,q}\left(  0,y\right)
+\left(  \mathfrak{B}_{1,q}+\frac{1}{2}\right)  \mathfrak{E}_{n-1,q}\left(
0,y\right)  . \label{cw3}%
\end{align}

\end{corollary}

The formulas (\ref{cw1})-(\ref{cw3}) are $q$-extension of the Cheon's main
result \cite{cheon}. Notice that $\mathfrak{B}_{1,q}=-\frac{1}{\left[
2\right]  _{q}},$ see \cite{manna}, and the extra term becomes zeo for
$q\rightarrow1^{-}.$

\begin{theorem}
\label{S-P2}For $n\in\mathbb{N}_{0}$, the following relationship%
\begin{align*}
\mathfrak{E}_{n,q}^{\left(  \alpha\right)  }\left(  x,y\right)   &  =%
%TCIMACRO{\dsum \limits_{k=0}^{n}}%
%BeginExpansion
{\displaystyle\sum\limits_{k=0}^{n}}
%EndExpansion
\frac{1}{m^{n-1}\left[  k+1\right]  _{q}}\left[  2%
%TCIMACRO{\dsum \limits_{j=0}^{k+1}}%
%BeginExpansion
{\displaystyle\sum\limits_{j=0}^{k+1}}
%EndExpansion
\left[
\begin{array}
[c]{c}%
k+1\\
j
\end{array}
\right]  _{q}\left(  \frac{1}{m}-1\right)  _{q}^{k+1-j}\mathfrak{E}%
_{j,q}^{\left(  \alpha-1\right)  }\left(  0,y\right)  \right. \\
&  -\left.
%TCIMACRO{\dsum \limits_{j=0}^{k+1}}%
%BeginExpansion
{\displaystyle\sum\limits_{j=0}^{k+1}}
%EndExpansion
\left[
\begin{array}
[c]{c}%
k+1\\
j
\end{array}
\right]  _{q}\left(  \frac{1}{m}-1\right)  _{q}^{k+1-j}\mathfrak{E}%
_{j,q}^{\left(  \alpha\right)  }\left(  0,y\right)  -\mathfrak{E}%
_{k+1,q}^{\left(  \alpha\right)  }\left(  0,y\right)  \right]  \mathfrak{B}%
_{n-k,q}\left(  mx,0\right)  ,\\
\mathfrak{E}_{n,q}^{\left(  \alpha\right)  }\left(  x,y\right)   &  =%
%TCIMACRO{\dsum \limits_{k=0}^{n}}%
%BeginExpansion
{\displaystyle\sum\limits_{k=0}^{n}}
%EndExpansion
\left[
\begin{array}
[c]{c}%
n\\
k
\end{array}
\right]  _{q}\frac{1}{m^{n}\left[  k+1\right]  _{q}}\left[  2%
%TCIMACRO{\dsum \limits_{j=0}^{k+1}}%
%BeginExpansion
{\displaystyle\sum\limits_{j=0}^{k+1}}
%EndExpansion
\left[
\begin{array}
[c]{c}%
k+1\\
j
\end{array}
\right]  _{q}m^{j}\mathfrak{E}_{j,q}^{\left(  \alpha-1\right)  }\left(
x,-1\right)  \right. \\
&  -\left.
%TCIMACRO{\dsum \limits_{j=0}^{k+1}}%
%BeginExpansion
{\displaystyle\sum\limits_{j=0}^{k+1}}
%EndExpansion
\left[
\begin{array}
[c]{c}%
k+1\\
j
\end{array}
\right]  _{q}m^{j}\mathfrak{E}_{j,q}^{\left(  \alpha\right)  }\left(
x,-1\right)  -m^{k+1}\mathfrak{E}_{k+1,q}^{\left(  \alpha\right)  }\left(
x,0\right)  \right]  \mathfrak{B}_{n-k,q}\left(  0,my\right)
\end{align*}
holds true between the $q$-Bernoulli polynomials and $q$-Euler polynomials.
\end{theorem}

\begin{proof}
The proof is based on the following identities%
\[
\left(  \frac{2}{e_{q}\left(  t\right)  +1}\right)  ^{\alpha}e_{q}\left(
tx\right)  E_{q}\left(  ty\right)  =\left(  \frac{2}{e_{q}\left(  t\right)
+1}\right)  ^{\alpha}E_{q}\left(  ty\right)  \cdot\frac{e_{q}\left(  \frac
{t}{m}\right)  -1}{t}\cdot\frac{t}{e_{q}\left(  \frac{t}{m}\right)  -1}%
e_{q}\left(  \frac{t}{m}mx\right)  ,
\]%
\[
\left(  \frac{2}{e_{q}\left(  t\right)  +1}\right)  ^{\alpha}e_{q}\left(
tx\right)  E_{q}\left(  ty\right)  =\left(  \frac{2}{e_{q}\left(  t\right)
+1}\right)  ^{\alpha}e_{q}\left(  tx\right)  \cdot\frac{e_{q}\left(  \frac
{t}{m}\right)  -1}{t}\cdot\frac{t}{e_{q}\left(  \frac{t}{m}\right)  -1}%
E_{q}\left(  \frac{t}{m}my\right)
\]
and similar to that of Theorem \ref{S-P1}.
\end{proof}

Next we discuss some special cases of Theorem \ref{S-P2}.

\begin{corollary}
For $n\in\mathbb{N}_{0}$, $m\in\mathbb{N}$ the following relationship%
\begin{align*}
\mathfrak{E}_{n,q}\left(  x,y\right)   &  =%
%TCIMACRO{\dsum \limits_{k=0}^{n}}%
%BeginExpansion
{\displaystyle\sum\limits_{k=0}^{n}}
%EndExpansion
\left[
\begin{array}
[c]{c}%
n\\
k
\end{array}
\right]  _{q}\frac{m^{-n}}{\left[  k+1\right]  _{q}}\left[  2%
%TCIMACRO{\dsum \limits_{j=0}^{k+1}}%
%BeginExpansion
{\displaystyle\sum\limits_{j=0}^{k+1}}
%EndExpansion
\left[
\begin{array}
[c]{c}%
k+1\\
j
\end{array}
\right]  _{q}m^{j}\left(  x-1\right)  _{q}^{j}\right. \\
&  \left.  -%
%TCIMACRO{\dsum \limits_{j=0}^{k+1}}%
%BeginExpansion
{\displaystyle\sum\limits_{j=0}^{k+1}}
%EndExpansion
\left[
\begin{array}
[c]{c}%
k+1\\
j
\end{array}
\right]  _{q}m^{j}\mathfrak{E}_{j,q}\left(  x,-1\right)  -m^{k+1}%
\mathfrak{E}_{k+1,q}\left(  x,0\right)  \right]  \mathfrak{B}_{n-k,q}\left(
0,my\right)
\end{align*}
holds true between the $q$-Bernoulli polynomials and $q$-Euler polynomials.
\end{corollary}

\begin{corollary}
\cite{luo2} For $n\in\mathbb{N}_{0}$, $m\in\mathbb{N}$ the following
relationship holds true.%
\begin{align*}
E_{n}\left(  x+y\right)   &  =%
%TCIMACRO{\dsum \limits_{k=0}^{n}}%
%BeginExpansion
{\displaystyle\sum\limits_{k=0}^{n}}
%EndExpansion
\frac{2}{k+1}\left(
\begin{array}
[c]{c}%
n\\
k
\end{array}
\right)  \left(  y^{k+1}-E_{k+1}\left(  y\right)  \right)  B_{n-k}\left(
x\right)  ,\\
E_{n}\left(  x+y\right)   &  =\sum_{k=0}^{n}\left(
\begin{array}
[c]{c}%
n\\
k
\end{array}
\right)  \frac{m^{k-n+1}}{k+1}\left[  2\left(  x+\frac{1-m}{m}\right)
^{k+1}-E_{k+1}\left(  x+\frac{1-m}{m}\right)  -E_{k+1}\left(  x\right)
\right]  B_{n-k}\left(  my\right)  .
\end{align*}

\end{corollary}

\begin{corollary}
For $n\in\mathbb{N}_{0}$ the following relationship holds true.%
\[
\mathfrak{E}_{n,q}\left(  x,y\right)  =%
%TCIMACRO{\dsum \limits_{k=0}^{n}}%
%BeginExpansion
{\displaystyle\sum\limits_{k=0}^{n}}
%EndExpansion
\left[
\begin{array}
[c]{c}%
n\\
k
\end{array}
\right]  _{q}\frac{2}{\left[  k+1\right]  _{q}}\left(  q^{\frac{1}{2}k\left(
k+1\right)  }y^{k+1}-\mathfrak{E}_{k+1,q}\left(  0,y\right)  \right)
\mathfrak{B}_{n-k,q}\left(  x,0\right)  .
\]

\end{corollary}

\begin{corollary}
For $n\in\mathbb{N}_{0}$ the following relationship holds true.%
\begin{align*}
\mathfrak{E}_{n,q}\left(  x,0\right)   &  =-%
%TCIMACRO{\dsum \limits_{k=0}^{n}}%
%BeginExpansion
{\displaystyle\sum\limits_{k=0}^{n}}
%EndExpansion
\left[
\begin{array}
[c]{c}%
n\\
k
\end{array}
\right]  \ \frac{2}{\left[  k+1\right]  _{q}}\mathfrak{E}_{k+1,q}%
\mathfrak{B}_{n-k,q}\left(  x,0\right)  ,\\
\mathfrak{E}_{n,q}\left(  0,y\right)   &  =-\sum_{k=0}^{n}\left[
\begin{array}
[c]{c}%
n\\
k
\end{array}
\right]  _{q}\frac{2}{\left[  k+1\right]  _{q}}\mathfrak{E}_{k+1,q}%
\mathfrak{B}_{n-k,q}\left(  0,y\right)  .
\end{align*}

\end{corollary}

These formulas are $q$-analogues of the formula of Srivastava and \'{A}.
Pint\'{e}r \cite{pinter}.

\section{$q$-Stirling Numbers and $q$-Bernoulli Polynomials}

In this section, we aim to derive several formulas involving the $q$-Bernoulli
polynomials, the $q$-Euler polynomials of order $\alpha,$ the $q$-Stirling
numbers of the second kind and $q$-Bernstein polynomials.

\begin{theorem}
\label{T:Stirling}Each of the following relationships holds true for the
Stirling numbers $S_{2}(n,k)$ of the second kind:%
\begin{align*}
\mathfrak{B}_{n,q}^{\left(  \alpha\right)  }\left(  x,y\right)   &
=\sum_{j=0}^{n}\left(
\begin{array}
[c]{c}%
mx\\
j
\end{array}
\right)  j!%
%TCIMACRO{\dsum \limits_{k=0}^{n-j}}%
%BeginExpansion
{\displaystyle\sum\limits_{k=0}^{n-j}}
%EndExpansion
\left[
\begin{array}
[c]{c}%
n\\
k
\end{array}
\right]  _{q}m^{j-n}\mathfrak{B}_{k,q}^{\left(  \alpha\right)  }\left(
0,y\right)  S_{2}\left(  n-k,j\right)  ,\\
\mathfrak{E}_{n,q}^{\left(  \alpha\right)  }\left(  x,y\right)   &
=\sum_{j=0}^{n}\left(
\begin{array}
[c]{c}%
mx\\
j
\end{array}
\right)  j!%
%TCIMACRO{\dsum \limits_{k=0}^{n-j}}%
%BeginExpansion
{\displaystyle\sum\limits_{k=0}^{n-j}}
%EndExpansion
\left[
\begin{array}
[c]{c}%
n\\
k
\end{array}
\right]  _{q}m^{j-n}\mathfrak{E}_{k,q}^{\left(  \alpha\right)  }\left(
0,y\right)  S_{2}\left(  n-k,j\right)  .
\end{align*}

\end{theorem}

The familiar $q$-Stirling numbers $S(n,k)$ of the second kind are defined by
\[
\frac{\left(  e_{q}\left(  t\right)  -1\right)  ^{k}}{\left[  k\right]  _{q}%
!}=\sum_{m=0}^{\infty}S_{2,q}\left(  m,k\right)  \frac{t^{m}}{\left[
m\right]  _{q}!},
\]
where $k\in\mathbb{N}.$ Next we give relationship between $q$-Bernstein basis
defined by Phillips \cite{phillips} and $q$-Bernoulli polynomials%
\[
b_{n,k}\left(  q;x\right)  :=x^{k}\left(  1-x\right)  _{q}^{n-k}.
\]

\begin{theorem}
\label{T:Bernstein}We have%
\begin{equation}
b_{n,k}\left(  q;x\right)  =x^{k}\sum_{m=0}^{n}\left[
\begin{array}
[c]{c}%
n\\
m
\end{array}
\right]  _{q}S_{2,q}\left(  m,k\right)  \mathfrak{B}_{n-m,q}^{\left(
k\right)  }\left(  1,-x\right)  . \label{bb1}%
\end{equation}

\end{theorem}

\begin{proof}
The proof follows from the following identities.%
\begin{align*}
\frac{x^{k}t^{k}}{\left[  k\right]  _{q}!}e_{q}\left(  t\right)  E_{q}\left(
-xt\right)   &  =\frac{x^{k}t^{k}}{\left[  k\right]  _{q}!}\sum_{n=0}^{\infty
}\frac{\left(  1-x\right)  _{q}^{n}t^{n}}{\left[  n\right]  _{q}!}=\sum
_{n=k}^{\infty}\left[
\begin{array}
[c]{c}%
n\\
k
\end{array}
\right]  _{q}\frac{x^{k}\left(  1-x\right)  _{q}^{n-k}t^{n}}{\left[  n\right]
_{q}!}\\
&  =\sum_{n=k}^{\infty}b_{n,k}\left(  q;x\right)  \frac{t^{n}}{\left[
n\right]  _{q}!}.
\end{align*}
and%
\begin{align*}
\frac{x^{k}t^{k}}{\left[  k\right]  _{q}!}e_{q}\left(  t\right)  E_{q}\left(
-xt\right)   &  =\frac{x^{k}\left(  e_{q}\left(  t\right)  -1\right)  ^{k}%
}{\left[  k\right]  _{q}!}\frac{t^{k}}{\left(  e_{q}\left(  t\right)
-1\right)  ^{k}}e_{q}\left(  t\right)  E_{q}\left(  -xt\right) \\
&  =x^{k}\sum_{m=0}^{\infty}S_{2,q}\left(  m,k\right)  \frac{t^{m}}{\left[
m\right]  _{q}!}\sum_{n=0}^{\infty}\mathfrak{B}_{n,q}^{\left(  k\right)
}\left(  1,-x\right)  \frac{t^{n}}{\left[  n\right]  _{q}!}\\
&  =x^{k}\sum_{n=0}^{\infty}\left(  \sum_{m=0}^{n}\left[
\begin{array}
[c]{c}%
n\\
m
\end{array}
\right]  _{q}S_{2,q}\left(  m,k\right)  \mathfrak{B}_{n-m,q}^{\left(
k\right)  }\left(  1,-x\right)  \right)  \frac{t^{n}}{\left[  n\right]  _{q}%
!}.
\end{align*}

\end{proof}

Finally, in their limit case when $q\rightarrow1^{-}$, these last result
(\ref{bb1}) would reduce to the following formula for the classical Bernoulli
polynomials $B_{n}^{\left(  k\right)  }\left(  x\right)  $ and the Bernstein
basis $b_{n,k}\left(  x\right)  =x^{k}\left(  1-x\right)  ^{n-k}:$%
\[
b_{n,k}\left(  x\right)  =x^{k}\sum_{m=0}^{n}\left[
\begin{array}
[c]{c}%
n\\
m
\end{array}
\right]  _{q}S_{2}\left(  m,k\right)  \mathfrak{B}_{n-m}^{\left(  k\right)
}\left(  1-x\right)  .
\]


\bigskip\bigskip

\bigskip

\begin{thebibliography}{99}                                                                                               %


\bibitem {andrew}G. E. Andrews, R. Askey and R. Roy Special functions, volume
71 of Encyclopedia of Mathematics and its Applications, Cambridge University
Press, Cambridge, 1999.

\bibitem {calitz1}L. Carlitz, $q$-Bernoulli numbers and polynomials, Duke
Math. J. 15 (1948) 987--1000.

\bibitem {calitz2}L. Carlitz, $q$-Bernoulli and Eulerian numbers, Trans. Amer.
Math. Soc. 76 (1954) 332--350.

\bibitem {calitz3}L. Carlitz, Expansions of $q$-Bernoulli numbers, Duke Math.
J. 25 (1958) 355--364.

\bibitem {cheon}G.-S. Cheon, A note on the Bernoulli and Euler polynomials,
Appl. Math. Lett. 16 (3) (2003) 365--368.

\bibitem {comtet}L. Comtet, Advanced Combinatorics: The Art of Finite and
Infinite Expansions (translated from French by J.W. Nienhuys), Reidel,
Dordrecht, Boston, 1974.

\bibitem {luo1}Q.-M. Luo, H.M. Srivastava, Some relationships between the
Apostol--Bernoulli and Apostol--Euler polynomials, Comput. Math. Appl. 51
(2006) 631--642.

\bibitem {luo2}Q.-M. Luo, Some results for the $q$-Bernoulli and $q$-Euler
polynomials, J. Math. Anal. Appl. 363 (2010) 7--18.

\bibitem {choi1}H.M. Srivastava, J. Choi, Series Associated with the Zeta and
Related Functions, Kluwer Academic Publishers, Dordrecht, Boston, London, 2001.

\bibitem {pinter}H.M. Srivastava, \'{A}. Pint\'{e}r, Remarks on some
relationships between the Bernoulli and Euler polynomials, Appl. Math. Lett.
17 (2004) 375--380.

\bibitem {sri1}Q.-M. Luo, H.M. Srivastava, $q$-extensions of some
relationships between the Bernoulli and Euler polynomials, Taiwanese Journal
Math., 15, No. 1, pp. 241-257, 2011.

\bibitem {cenkci1}M. Cenkci and M. Can, Some results on $q$-analogue of the
Lerch Zeta function, Adv. Stud. Contemp. Math., 12 (2006), 213-223.

\bibitem {cenkci2}M. Cenkci, M. Can and V. Kurt, $q$-extensions of Genocchi
numbers, J. Korean Math. Soc., 43 (2006), 183-198.

\bibitem {cenkci3}M. Cenkci, V. Kurt, S. H. Rim and Y. Simsek, On
$(i,q)$-Bernoulli and Euler numbers, Appl. Math. Lett., 21 (2008), 706-711.

\bibitem {choi2}J. Choi, P. J. Anderson and H. M. Srivastava, Some
$q$-extensions of the Apostol-Bernoulli and the Apostol-Euler polynomials of
order $n$, and the multiple Hurwitz Zeta function, Appl. Math. Comput., 199
(2008), 723-737.

\bibitem {choi3}J. Choi, P. J. Anderson and H. M. Srivastava, Carlitz's
$q$-Bernoulli and $q$-Euler numbers and polynomials and a class of $q$-Hurwitz
zeta functions, Appl. Math. Comput., 215 (2009), 1185-1208.

\bibitem {kim1}T. Kim, Some formulae for the $q$-Bernoulli and Euler
polynomial of higher order, J. Math. Anal. Appl., 273 (2002), 236-242.

\bibitem {kim2}T. Kim, $q$-Generalized Euler numbers and polynomials, Russian
J. Math. Phys. 13 (2006), 293-298.

\bibitem {kim3}T. Kim, On the $q$-Extension of Euler numbers and Genocchi
numbers, J. Math. Anal. Appl., 326 (2007), 1458-1465.

\bibitem {kim4}T. Kim, $q$-Bernoulli numbers and polynomials associated with
Gaussian binomial coefficients, Russian J. Math. Phys., 15 (2008), 51-57.

\bibitem {kim5}T. Kim, The modified $q$-Euler numbers and polynomials, Adv.
Stud. Contemp. Math., 16 (2008), 161-170.

\bibitem {kim6}T. Kim, L. C. Jang and H. K. Pak, A note on $q$-Euler numbers
and Genocchi numbers, Proc. Japan Acad. Ser. A Math. Sci., 77 (2001), 139-141.

\bibitem {kim7}T. Kim, Y.-H. Kim and K.-W. Hwang, On the $q$-extensions of the
Bernoulli and Euler numbers. related identities and Lerch zeta function, Proc.
Jangjeon Math. Soc., 12 (2009), 77-92.

\bibitem {kim8}T. Kim, S.-H. Rim, Y. Simsek and D. Kim, On the analogs of
Bernoulli and Euler numbers, related identities and zeta and $L$-functions, J.
Korean Math. Soc., 45 (2008), 435-453.

\bibitem {ozden}H. Ozden and Y. Simsek, A new extension of $q$-Euler numbers
and polynomials related to their interpolation functions, Appl. Math. Lett.,
21 (2008), 934-939.

\bibitem {manna}O-Yeat Chan, D. Manna, A new $q$-analogue for bernoulli
numbers, Preprint, oyeat.com/papers/qBernoulli-20110825.pdf

\bibitem {phillips}G. M. Phillips, On generalized Bernstein polynomials.
Numerical analysis, 263--269, World Sci. Publ., River Edge, NJ, 1996.

\bibitem {tyoo1}C. S. Ryoo, J. J. Seo and T. Kim, A note on generalized
twisted $q$-Euler numbers and polynomials, J. Comput. Anal. Appl., 10 (2008), 483-493.

\bibitem {simsek1}Y. Simsek, q-Analogue of the twisted $l$-series and
$q$-twisted Euler numbers, J. Number Theory, 110 (2005), 267-278.

\bibitem {simsek2}Y. Simsek, Twisted $(h,q)$-Bernoulli numbers and polynomials
related to twisted $(h,q)$-zeta function and L-function, J. Math. Anal. Appl.,
324 (2006), 790-804.

\bibitem {simsek3}Y. Simsek, Generating functions of the twisted Bernoulli
numbers and polynomials associated with their interpolation functions, Adv.
Stud. Contemp. Math., 16 (2008), 251-278.

\bibitem {sri}H. M. Srivastava, T. Kim and Y. Simsek, q-Bernoulli numbers and
polynomials associated with multiple q-Zeta functions and basic L-series,
Russian J. Math. Phys., 12 (2005), 241-268.
\end{thebibliography}
\end{document}